%% file: RS-Optcon_2024.tex
\documentclass[11pt]{amsart}

\input{macro}

\usepackage{amsmath,amssymb,latexsym,amsthm,mathdots,yhmath,xcolor}

\usepackage{cancel}
\usepackage{blindtext}
\usepackage{changes}

\newtheorem{theorem}{Theorem}[section]
\newtheorem{lemma}[theorem]{Lemma}
\newtheorem{proposition}[theorem]{Proposition}
\newtheorem{corollary}[theorem]{Corollary}

\theoremstyle{remark}
{
    \newtheorem{definition}[theorem]{Definition}
    
    \newtheorem{remark}[theorem]{Remark}

}

\DeclareMathAlphabet{\mathpzc}{OT1}{pzc}{m}{it}

\newcommand{\Vspace}{L^2(\mu,\Omega;\cR^n)}

\newcommand{\Uad}{\mathcal{U}_{\rm ad}}


\begin{document}

\title[HJB for Control Problems with Uncertainty]{Dynamic Programming Principle and Hamilton-Jacobi-Bellman Equation for Optimal Control Problems with Uncertainty}

\author[M.S. Aronna]{M. Soledad Aronna}
\address{M.S. Aronna\\ Escola de Matem\'atica Aplicada, Funda\c c\~ ao Getulio Vargas, Praia de Botafogo 190,   22250-900 Rio de Janeiro - RJ, Brazil}
\email{soledad.aronna@fgv.br}

\author[M. Palladino]{Michele Palladino}
\address{M. Palladino\\ University of L'Aquila - Department of Information Engineering, Computer Science and Mathematics (DISIM), via Vetoio, 67100,, L'Aquila, Italy}
\email{michele.palladino@univaq.it}


\author[O. Sierra]{Oscar Sierra}
\address{O. Sierra\\ Escola de Matem\'atica Aplicada, Funda\c c\~ ao Getulio Vargas, Praia de Botafogo 190,   22250-900 Rio de Janeiro - RJ, Brazil}
\email{oscar.fonseca@fgv.br}

\begin{abstract} 
We study the properties of the value function associated with an optimal control problem with uncertainties, known as {\em average} or {\em Riemann-Stieltjes} problem. Uncertainties are assumed to belong to a compact metric probability space, and appear in the dynamics, in the terminal cost and in the initial condition, which yield an infinite-dimensional formulation.
By stating the problem as an evolution equation in a Hilbert space, we show that the value function is the unique lower semi-continuous proximal solution of the Hamilton-Jacobi-Bellman (HJB) equation. Our approach relies on invariance properties and the dynamic programming principle.
\end{abstract}

\maketitle

\footnotetext[1]{The first and third authors acknowledge the support of FAPERJ (Rio de Janeiro, Brazil) via processes SEI-260003/004623/2021 and SEI-260003/000175/2024; and CNPq (Brazil) through process code 312407/2023-8. }


\section{Introduction}\label{Introduction}
This work studies the properties of the value function and the Hamilton-Jacobi-Bellman equation related to a class of optimal control problems that considers uncertainty in the dynamics and initial state. From a practical point of view, the framework covers models with stochasticity in the parameter values and in the  initial conditions. The uncertain values are assumed to belong to a probability space, and the cost involves the average of a terminal cost computed among the corresponding probability measure. This framework is known in the literature as {\em average optimal control} \cite{zuazua2014} or  {\em Riemman-Stieltjes} problems \cite{Ross2015,Khalil2017Thesis,BettiolKhalil}.

The main contribution of the work is providing a theoretical framework by stating the problem in a (possibly infinite dimensional) Hilbert space for the time dependent state variable and proving that the associated value function is the unique lower semi-continuous solution of a Hamilton-Jacobi-Bellman (HJB) equation, under mild regularity assumptions on the governing dynamics (it is required to be merely measurable in time). More specifically, we study the following parametrized Riemann–Stieltjes problem {with respect to} the initial time $s\in[0,T]$  and the initial state,  the latter being an $L^2$-function of the parameter $\omega,$ this is, the initial condition is represented by elements $\varphi$ in $ \Vspace$:
\begin{equation*} 
(P)_{s,\varphi} \parbox{.9\textwidth}{
\begin{align}
&\min_{u} \int_\Omega g(x(T,\omega),\omega) \, d\mu(\omega), \nonumber \\
&\text{s.t.} \nonumber \\
&\begin{cases}\label{dynamics}
\dot{x}(t,\omega) = f(t,x(t,\omega),u(t),\omega), & \text{a.e. } t\in [s,T],  \ \omega\in \Omega, \\
x(s,\omega) = \varphi(\omega), &  \omega\in \Omega, \\
u(t) \in U(t), & \text{a.e. } t\in [s,T].
\end{cases}
\end{align}
}\nonumber
\end{equation*}
Here $u:[s,T]\rightarrow \cR^m$ is the control function, which belongs to the {\em admissible controls set}
\benl
\mathcal{U}_{\rm ad}[s,T] := \left\{ u \in  L^\infty(s,T;\cR^m): u(t)\in U(t) \quad\text{a.e. } t\in [s,T]\right\},
\eenl 
and the triple  $(\Omega,d_{\Omega},\mu)$ is a metric measure space.

This type of problem has been addressed using various approaches, including different theoretical frameworks and computational techniques.  In \cite{RossKarpenko2015}, the authors propose models from aerospace engineering in the Lebesgue-Stieltjes framework, and in \cite{Ross2015}, they formally derive a Pontryagin Maximum Principle for that class of average problems. In  \cite{Phelps2016}, the authors  provide a computational framework for this uncertain dynamic optimization problems, providing consistent approximation techniques. 
The authors in \cite{PesarePalladinoFalcone2021} deal with the linear quadratic case, showing convergence results in model-based reinforcement learning scenarios. In \cite{Scagliotti2023}, this class of problems is referred to as {\em ensemble optimal control}. Latter work establishes convergence results and numerical algorithms for their solution. It is worth highlighting , that our framework involves an average cost, while other authors have studied the worst-case scenario formulation (see {\em e.g.}  \cite{Vinter2005} for {\em minimax optimal control}).

The problem $(P)_{s,\varphi}$ was also studied in \cite{BettiolKhalil} when the initial condition is constant-valued, {\em i.e.} $\varphi\equiv x_0 \in \mathbb{R}^n$, and with an additional averaged right end-point constraint. The authors state necessary conditions for optimality by considering, for each control $u$, the set of pointwise trajectories $\{x(\cdot,\omega)\in W^{1,1}([0,T]:\mathbb{R}^n): \omega \in \Omega\}$  
 satisfying the constraints. Here we extend the set of constraints to an ODE defined in a Hilbert space, that is, given an initial condition $\varphi\in \Vspace$ and a control variable $u\in \Uad[s,T]$, the corresponding feasible trajectory $x(\cdot,\cdot)$ belongs to $ C([s,T]: \Vspace)$ and satisfies  \eqref{dynamics}. 

The {\em value function} for the Riemann-Stieltjes problem $(P)_{s,\varphi}$ is the extended-valued mapping given by
\benl
\begin{aligned}
V \colon [0,T]\times \Vspace & \rightarrow \mathbb{R}\cup\{\pm \infty\}\\
(s,\varphi) & \mapsto V(s,\varphi) := \inf_{u\in \Uad[s,T]} \,(P)_{s,\varphi}.
\end{aligned}
\eenl
When $V$ is Fréchet differentiable with respect to $(t,\varphi)$ in $ (0,T)\times \Vspace,$
it satisfies the HJB equation
\be\label{HJB}
 \begin{cases}
    -\left[ V_{t}(t,\varphi)+H(t,\varphi,V_{\varphi}(t,\varphi))\right]=0, \\ \\
    V(T,\varphi)=\displaystyle \int_{\Omega} g\left(\varphi(\omega), \omega\right) d \mu(\omega),
\end{cases}   
\ee
where the Hamiltonian $H:[0,T]\times \Vspace \times \Vspace^* \rightarrow \mathbb{R}$ is given by 
\be\label{Hamiltonian}
    H(t,\varphi, p):=\inf_{u\in {U(t)}} \left\langle p(\cdot), f(t, \varphi(\cdot), u, \cdot) \right\rangle,
\ee
and $(V_t,V_{\varphi})\in (\cR \times \Vspace)^*$ denotes the  Fréchet derivative  of $V$ with respect to $(t,\varphi)\in (0,T)\times \Vspace$ 
Here, we deal with nonsmooth solutions of the HJB equation, {\em i.e.} we consider \eqref{HJB} for $V$ merely lower semi-continuous (l.s.c.)  and we adopt the notions of {\em proximal subdifferential} (denoted by $\partial_PV$) and {\em proximal solutions} \cite{ClarkeLedyaev1994}.  To prove that $V$ is the unique l.s.c. proximal solution of \eqref{HJB}, we rely on results of {\em invariance} from \cite{Donchev}. Latter work consists of an extension to the infinite-dimensional case of the work of Clarke \textit{et al.} \cite{Clarke1995QualitativePO}, where the HJB equation is related to the concepts of weak and strong invariance (often called {\em viability}  and invariance, respectively \cite{AubCelBook}).

 Problems where the state space is finite-dimensional and the cost function is merely lower semi-continuous have been extensively studied.  For instance, Frankowska in \cite{Frankowska93} addresses the case when the dynamics is continuous in time; later in \cite{FRANKOWSKA1995}, the result is extended to measurable in time dynamics. 
In both cases, the approach is to prove the invariance properties of an associated autonomous system along with the well-known {\em Dynamic Programming Principle.} In \cite{FRANKOWSKA1995}, the authors show that the nonsmooth solutions of the HJB equation belong to the subdifferential of the value function and are characterized by means of the {\em contingent cone} \cite{AubCelBook},\cite{AubFraBook} to the epigraph of the value function. Such a characterization fails when the space is infinite dimensional (see \cite[Proposition 6.4.8]{AubFraBook}). In our framework, we use analogous objects from the proximal analysis known as  {\em proximal subdifferentials} (see Section \ref{DP.and.IP}), which are characterized by the elements of the {\em normal proximal cone} to the epigraph of the value function; these gradients will be involved in the concept of {\em proximal solution} of the HJB equation.
The results of existence and uniqueness of the HJB equation provided in \cite{Frankowska93,FRANKOWSKA1995}   rely on the existence of a minimum of the cost function, which requires the compactness of the set of trajectories. In general, this property does not hold for infinite-dimensional spaces. However, in the semilinear case, when defined on separable Banach spaces, the lack of compactness of the set of trajectories can be overcome by  a compactness assumption on the semigroup related to the differential operator and application of the  Arzelà-Ascoli Theorem \cite{FRANKOWSKA1990, Fattorini1999,liYong}.

In the present work, the cost is given by an integral functional in which the lower semi-continuity is inherited by the function $g$,
 and the dynamics are purely nonlinear. The compactness  of the set of trajectories is addressed by introducing an assumption on the measure $\mu$ and ensuring some regularity properties for the mapping $\omega \mapsto \int_{s}^{t} f(\sigma,\varphi(\omega),u(\sigma),\omega)d\sigma$. Additionally, we employ a Rellich-Kondrakov-type theorem \cite{Hajlasz2000},\cite{Agnieszka}. 
 
The paper is outlined as follows. Section \ref{sec.DPP}  includes properties on the behavior of the trajectories and the statement and proof of the Dynamic Programming Principle. Section \ref{sec.op.traj} is dedicated to the study of the compactness properties of the set of trajectories, characterizing the lower semi-continuity of the value function, and showing the existence of minimizers. Section \ref{DP.and.IP} is devoted to proving that the value function is the unique proximal solution of the HJB equation.  Finally, the paper ends with an appendix, presenting some technical results needed in our proofs. 

\subsection{Notations, basic assumptions and definitions}\label{notation}

Throughout the paper, we impose the following set of hypotheses, which we refer to hereafter with {\bf (H)}:
\begin{itemize}
\item[(i)]  $(\Omega,d_{\Omega},\mu)$ is a compact metric measure space,   $\mu$ being a finite Radon measure. 
\item[(ii)] $U:[0,T] \leadsto {\bf U}$ is an upper semi-continuous set-valued function taking nonempty compact values, included in a compact set $ {\bf U} \subset \cR^m$. 
\item[(iii)] $f\colon [0,T] \times \cR^n \times \cR^m \times \Omega \to \cR^n$ is measurable w.r.t. $t$ and continuous on the other variables, and there exist constants $c,k>0$ such that, for all  $x\in \cR^n,u\in \cR^m,$ $\omega\in \Omega,$ a.e. $t\in (0,T),$ $f$ satisfies
\be
\label{fbounded}
|f(t,x,u,\omega)| \leq c (1+|x|),
\ee
and
\be
\label{fLipschitz}
|f(t,x,u,\omega)-f(t,x',u,\omega)|\leq k|x-x'|.
\ee
\item[(iv)] The map $g:\mathbb{R}^n \times \Omega \rightarrow \mathbb{R}$ is such that
 $$ 
 g(x, \omega) \geq a(\omega)-b|x|^2, \ \text{for every} \ \  (x,\omega) \in \mathbb{R}^n\times \Omega,$$ 
 for some $a \in L^1(\mu,\Omega):=
 \{a:\Omega\rightarrow \cR : \int_{\Omega}|a(\omega)|\,d\mu(\omega)<\infty\},
$ and $ b \geq 0.$
Additionally, $g$ satisfies the following two conditions:

\begin{itemize}
\item[$\bullet$] $g$ is a $\mathcal{L}\otimes\mu$ -measurable function, $\mathcal{L}$ here denoting the Lebesgue measure
\vspace{0.15cm}
\item[$\bullet$] $g(\cdot,\omega)$ is lower semi-continuous for every fixed $\omega \in \Omega,$
\end{itemize}
\end{itemize}
The conditions on $g$ in item (iv) above are equivalent to stating that $g$ is a  {\em normal integrand} 
\cite[p. 195]{castaing1975convex}.


The values of our state variables will be taken in the space
$$
\Vspace:=\Big\{\varphi:\Omega\rightarrow \cR^n : \int_{\Omega}|\varphi(\omega)|^2\,d\mu(\omega)<\infty\Big\},
$$
which is a Hilbert space endowed with the scalar product
$$
\langle \varphi,\psi\rangle:= \int_{\Omega}\varphi(\omega)\cdot \psi(\omega)\,d\mu(\omega) \quad \text{for} \quad \varphi,\psi \in \Vspace.
$$
Throughout this paper, we will frequently use the notation $
\|\cdot \|_{\Vspace}$ and $\|\cdot \|_{L^2}
$
to refer to the norm associated with the scalar product above.

If $X$ is a vector space we will  $\mathcal{B}_{r}(x)$ denote the open ball centered at $x$ with radius $r>0.$
We call {\em process} a pair $(x, u)$  composed of an absolutely continuous function $x \in W^{1,1}([s,T]:\Vspace) $ and a measurable function $u:[s,T]\rightarrow \cR^m$, which together satisfy the constraints \eqref{dynamics} of problem $(P)_{s,\varphi}.$ Additionally, $x$ is referred as {\em trajectory} associated to the {\em control} $u$. 
If there exists a  process $(\bar{x},\bar{u})$ solving the problem $(P)_{s,\varphi},$ then it is called   {\em optimal pair}, and we will refer to $\bar{x}$ and $\bar{u}$ as an {\em optimal trajectory} and {\em control}, respectively. A trajectory $x(t,\cdot)$ for the problem $(P)_{s,\varphi}$ can be written  in the form of the following evolution integral equation
\be\label{trajectory}
x(t,\cdot)=\varphi(\cdot)+\int_{s}^{t} f(\sigma,x(\sigma,\cdot),u(\sigma);\cdot) \,d\sigma.
\ee
It is clear (see {\em e.g.} \cite[Ch.6, Theorem 1.2.]{Pazy}) that under ${\bf (H)},$ for any $(\varphi, u) \in  \Vspace \times \Uad[s,T],$ there exists a unique solution $x(\cdot,\cdot)\in C([s,T]: \Vspace)$ of \eqref{trajectory}.

The {\em reduced cost} of $(P)_{s,\varphi}$ will be defined as
$$
J_{[s,T]}(\varphi,u) := \int_\Omega g(x_{\varphi,u}(T,\omega),\omega)d\mu(\omega),
$$ 
where $x_{\varphi,u}$ is the trajectory corresponding to the initial condition $x(s,\cdot)=\varphi(\cdot)$, and the choice of control $u$. With these considerations, the value function can be written as
\be\label{vf}
V(s,\varphi):=\inf_{u\in \Uad[s,T]} J_{[s,T]}(\varphi,u) =  \inf_{u\in \Uad[s,T]} \int_\Omega g(x_{\varphi,u}(T,\omega),\omega)d\mu(\omega).
\ee

\section{Dynamic Programming Principle}\label{sec.DPP}

In this section, the Dynamic Programming Principle for problem family  $(P)_{s,\varphi}$ is stated. 
This property is the key tool needed to obtain the results of Section \ref{DP.and.IP}, it establishes that the value function is nondecreasing over trajectories of \eqref{dynamics} and constant over optimal trajectories. In other words, optimal trajectories are the characteristic curves for the HJB equation.

 The following result embraces the properties of the trajectories of \eqref{dynamics}.  Here we will write $x_{s,\varphi}$  for the solution of \eqref{dynamics} with initial time $s\in [0, T]$ and datum $\varphi \in \Vspace$. In what follows, $C(T)$  is a positive constant depending on $T$, which can vary depending on the context. 
\begin{lemma}\label{x.ineq}
     Let assumption {\bf (H)} hold. Then, for any $s$ and $\tau$ with $0 \leq s \leq \tau \leq T,$  $\varphi, \bar{\varphi} \in \Vspace$, and any control $u \in \Uad[s,T]$, the following inequalities hold
   
    $$\begin{aligned}
 &\textit{ 1.}) \
\|x_{s, \varphi}(t)\|_{L^ 2} \leq C(T)e^{c(t-s)}\left(\mu(\Omega)^{1/2}+\|\varphi\|_{L^2} \right), \quad t \in[s, T] . \\ 
 &\textit{ 2.}) \ 
\|x_{s, \varphi}(t)-x_{s, \bar{\varphi}}(t)\|_{L^2}  \leq e^{k(t-s)}\|\varphi-\bar{\varphi}\|_{L^2} , \quad t \in[s, T] .
\\ 
 &\textit{ 3.}) \ \|x_{\tau, \varphi}(t)-x_{s, \varphi}(t)\|_{L^2}  \leq \\
 &\hspace{3cm}C(T)e^{k(t-\tau)} \big(\mu(\Omega)^{1/2}+\|\varphi\|_{L^2} \big)(\tau-s), \quad  t \in[\tau, T].
\\ 
&\textit{4.}) \
\|x_{s, \varphi}(t)- x_{s, \varphi}(\tau)\|_{L^2}  \leq C(T)e^{c(t-s)} \big(\mu(\Omega)^{1/2}+\|\varphi\|_{L^2} \big)(t-\tau), \ t \in[s, T] .
 \end{aligned}
 $$
\end{lemma}
\begin{proof}
It follows directly from assumptions \eqref{fbounded}, \eqref{fLipschitz}, and a standard application of Grönwall's Lemma (see for instance \cite[Ch.6, Lemma 2.1.]{liYong}). 
\end{proof}

The following result holds. 
\begin{theorem}[Dynamic Programming Principle]
\label{DPP} Under the  hypothesis {\bf (H)}, the following conditions hold true:
\begin{itemize}
\item[i)] for every $\varphi \in \Vspace$, for every and $s_{1},s_{2}\in [0,T]$, with $s_{1}\leq s_{2}$, one has
\begin{equation}\label{DPP_1}
V(s_{1},\varphi) \leq V(s_{2},x_{\varphi,u}(s_{2};\,\cdot)),
\end{equation}
for any measurable control $u\in \Uad[s_1,s_2],$ where $x_{\varphi,u}$ is the trajectory with initial condition $x_{\varphi,u}(s_1,\cdot)=\varphi(\cdot)$ and associated to control $u\in \Uad[s_1,s_2];$
\item[ii)] for every $\varphi \in \Vspace$  and $s_{1},s_{2}\in [0,T]$, with $s_{1}\leq s_{2}$, one has
\begin{equation}\label{DPP_2}
V(s_{1},\varphi) = \inf_{u\in \Uad[s_1,T]} V(s_{2},x_{\varphi,u}(s_2,\cdot)).
\end{equation}
\end{itemize}
\end{theorem}

\begin{proof} See Appendix \ref{ProofDPP}.
\end{proof}


\section{Existence of Optimal controls}\label{sec.op.traj}

In this section, it is proved that the problem $(P)_{s,\varphi}$ admits at least one solution. The approach involves using differential inclusions to show the compactness of the set of trajectories which, together with the lower semi-continuity of the functional ${\mathcal J}:\Vspace \rightarrow \cR,$ given by
\begin{equation}
\label{calJ}
{\mathcal J}(\varphi):=\int_{\Omega}g(\varphi(\omega),\omega)d\mu(\omega),
\end{equation} 
guarantees the result. At this point, it is important to mention that, in view of Proposition \ref{multifunctionU} in the Appendix, assumption {\bf (H)}(ii) on the regularity properties of the set-valued mapping $U$ ensures its pseudo-continuity ({\em Definition} \ref{pseudocontinuous}). We make use of this property of $U$ in the proof of closedness of the set of trajectories.

Additionally, we consider the set-valued map 
\begin{equation}
\label{Fdef}
\begin{aligned}
F: [0, T] \times L^2(\mu,\Omega;\mathbb{R}^n) &\leadsto L^2(\mu,\Omega;\mathbb{R}^n) \\
(t,\varphi) &\mapsto F(t,\varphi(\cdot)):=f(t, \varphi(\cdot), U(t),\cdot )
\end{aligned}
\end{equation}
and the associated differential inclusion
\be\label{inclu}
    \dot{x}(t,\cdot) \in F(t,x(t,\cdot))  \ \ \ \text{a.e.}
    \ t \in [s,T].
\ee
Fillipov's Lemma (see Corollary \ref{selection} in the Appendix) allows us to reduce the differential equation 
\be\label{dif.eq}
 \dot{x}(t,\cdot)= f(t, x(t,\cdot), u(t),\cdot), \quad u(t) \in U(t)  \ \ \ \text{a.e.}
    \ t \in [s,T]
\ee
to the differential inclusion \eqref{inclu}. More precisely, Fillipov's Lemma states that \eqref{inclu} and \eqref{dif.eq} have the same set of trajectories.  We will denote the set of trajectories of \eqref{inclu} with initial data $\varphi\in\Vspace$ at time $s\in[0,T]$  as 
\benl
\begin{aligned}
    S_{[s,T]}(\varphi):= \left\{ x \in C([s,T]: L^2(\mu,\Omega;\mathbb{R}^n)) : \  x  \ \text{solves  \eqref{inclu},} \ x(s,\cdot )=\varphi(\cdot) \right\}.
    \end{aligned}
\eenl

\subsection{Compactness of the set of trajectories}
We are interested in the compactness properties of the trajectories of \eqref{inclu} with fixed initial time and datum. {For purely nonlinear dynamics defined in a Banach space, this property can be obtained by assuming an inequality that involves a measure of non-compactness \cite{Tolst}, a situation not treated here. On the other hand,} for semilinear evolutionary distributed parameter systems in a separable Banach space $X$, the compactness of such a set is achieved through compactness assumptions either on 
a ${C_0}$-semigroup $S$ or on the nonlinear part. In \cite[page 110]{liYong}, this result is obtained assuming that the semigroup $S$ is compact, and in \cite{FRANKOWSKA1990,Frankowska1992}, it is achieved when at least one of the following two conditions is satisfied: compactness of $S$ or inclusion of the nonlinear part in a compact set.
When $S$ is compact or the nonlinear part  maps into a compact set, the first step to obtain  the compactness of the set of trajectories is proving the  same property for the operator 
$$
\mathcal{S}:L^{p}([s,T]:X)\rightarrow C([s,T]:X), \quad p>1,
$$ 
given by 
$$\mathcal{S}(f(\cdot)):=\int_{s}^{\cdot}S(\cdot -\sigma)f(\sigma)d\sigma,$$
which corresponds to the integral component of the mild solution of the associated semilinear system \cite{liYong}. The compactness of $\mathcal{S}$ yields the relative compactness of the set of trajectories. Then, the application of Mazur's Lemma to the differential inclusion approach establishes the closedness of $S_{[s,T]}(\varphi)$. Our method here consists of tailoring and extending the described approaches to our framework. Thus, for the nonlinear equation \eqref{dif.eq}, we will introduce additional conditions on the measure and dynamics to establish the compactness of the analogous operator $$\mathcal{I}:C([s,T]: \Vspace)\times \Uad[s,T] \rightarrow C([s,T]: \Vspace)$$ defined by
\be\label{OP.INT}
\mathcal{I}(x,u):=\int_{s}^{\cdot} f(\sigma,x(\sigma,\cdot),u(\sigma),\cdot)d\sigma, \ \text{where} \ (x,u) \ \text {is a process.} 
\ee
  Thus, if $\{(x_k,u_k)\}_k$ is a bounded sequence of control processes, the corresponding sequence $\{\mathcal{I}(x_k,u_k)\}_k$ contains a subsequence converging to some function $F^*$ in $ C([s,T]: \Vspace),$ then, the Mazur's Lemma and Fillipov's Lemma will ensure that $F^*$ is the integral part of a trajectory. Following Fattorini \cite{Fattorini1999}, that property means that the control system is {\em trajectory complete}.

We consider the following hypothesis ${\bf (H_{\mu})}$ consisting of the following two conditions:
\be\label{Hyp.mu}
\text{for any}\ \ r>0, \ h(r)=\inf \{\mu(\mathcal{B}_{r}(\omega)): \omega \in \Omega\}>0.
\ee 

There exists a modulus of continuity $\theta_{f}(\cdot)$ such that, for all $\omega_{1},\omega_{2} \in \Omega,$
\be\label{fwregular}
\int_{0}^{T}\sup_{u\in U(t), \ x,y\in\cR^n}|f(t,x,u,\omega_{1})-f(t,y,u,\omega_{2})|dt\leq \theta_{f}(d_{\Omega}(\omega_{1},\omega_{2})).
\ee
Additionally, the following property is introduced: 
$$\mathrm{\textbf{(C)}}\quad F(t,\varphi) \text{ takes convex values for each } (t,\varphi)\in [0,T]\times \Vspace. $$
Also, as noted in Remark \ref{F.compact} in the Appendix, $F(t,\varphi)$ takes compact values  for almost all $t\in [s,T]$ and all $\varphi\in\Vspace.$


\begin{remark}
    When ${\rm supp}\, \mu$ is finite, the resulting optimal control problem is finite-dimensional, so compactness of trajectories can be proved in a standard way (see {\em e.g.} \cite[Theorem 23.2]{clarke2013functional}).
\end{remark}

\begin{remark}
   If  $\Omega \subset \cR^d,$ whenever $\mu$ is absolutely continuous with respect to the Lebesgue measure in $\Omega$, and the latter coincides with the support of $\mu$,  hypothesis \eqref{Hyp.mu} holds.
\end{remark}

Henceforth, in cases where the dependence on the parameter variable $\omega$ can be omitted, we will adopt the notation $\varphi:=\varphi(\cdot)$ and $x(t):=x(t,\cdot)$ for the initial datum and trajectories, respectively,  and $f(t,x(t),u(t)):=f(t,x(t,\cdot ),u(t),\cdot )$ for the dynamics.

Let $\{x_k\}_k \subset C([s,T]: \Vspace)$  be a sequence of solutions of \eqref{dif.eq}  with $x_k(s )=\varphi \in \Vspace$ and controls  $u_k \in \Uad[s,T].$  Set
\be\label{Fk}
F_k(t ):=\int_{s}^{t} \Bar{f}_{k}(\sigma ) d\sigma,
\ \text{where} \ \Bar{f}_{k}(\sigma ):= f(\sigma,x_k(\sigma ),u_k(\sigma) ).
\ee
\begin{proposition}\label{Fkprecompact}
Let us assume that ${\bf (H)}$ and ${\bf (H_{\mu})}$ hold. Then, for all $t\in [s,T],$ the sequence $\{F_k(t )\}_k$ introduced in \eqref{Fk} is relatively compact in $\Vspace.$
\end{proposition}
\begin{proof} We will apply  Theorem \ref{compactL2} (see Appendix \ref{AppCompactness}) in $\Vspace$. Inequality \eqref{fbounded} and Lemma \ref{x.ineq}  guarantee that
    \be\label{fkbounded}
        ||\Bar{f}_{k}(\sigma) ||_{\Vspace} \leq C\left(\mu(\Omega)^{1/2}+\|\varphi\|_{\Vspace}  \right) 
    \ee 
    a.e. $\sigma \in [s,T]$ and all $k.$ It follows that, for any $t\in [s,T],$ the  sequence $\{F_k(t )\}_k$ is bounded  in $\Vspace.$ On the other hand, from \eqref{fwregular} we observe that
    {\small
    \be\label{mFk}
    \begin{aligned}
      & \int_{\Omega} \mid F_k(t,\omega)-(F_k(t ))_{\mathcal{B}_{r}(\omega)} \mid^2 d\mu(\omega) 
       \\
       &= \int_{\Omega} \bigg| \int_{s}^{t} \Bar{f}_{k}(\sigma,\omega)d\sigma - \frac{1}{\mu(\mathcal{B}_{r}(\omega))}\int_{\mathcal{B}_{r}(\omega)} \left[ \int_{s}^{t} \Bar{f}_{k}(\sigma,\omega')d\sigma \right] \, d\mu(\omega') \bigg|^2 d\mu(\omega) 
       \\
       &= \int_{\Omega} \bigg| \frac{1}{\mu(\mathcal{B}_{r}(\omega))}\int_{\mathcal{B}_{r}(\omega)}  \int_{s}^{t} \left[\Bar{f}_{k}(\sigma,\omega)-  \Bar{f}_{k}(\sigma,\omega')\right]d\sigma  \, d\mu(\omega') \bigg|^2 d\mu(\omega) 
       \\
       &\leq \int_{\Omega} \bigg| \frac{1}{\mu(\mathcal{B}_{r}(\omega))}\int_{\mathcal{B}_{r}(\omega)}  \theta_{f}(d_{\Omega}(\omega,\omega')) \, d\mu(\omega') \bigg|^2 d\mu(\omega)
       \\
       &\leq \mu(\Omega)[\theta_{f}(r)]^2.
    \end{aligned}
    \ee
    }
Thus, one gets
 $$\sup_{k}\int_{\Omega} \mid F_k(t,\omega)-(F_k(t ))_{\mathcal{B}_{r}(\omega)} \mid^2 d\mu(\omega) \to 0, \quad \text{as} \ r\rightarrow 0,$$
 and the proof is concluded.
\end{proof}

\begin{remark}
The property \eqref{fwregular} implies the continuity of the maps
$$
\omega \mapsto F_k(t,\omega), \ \ \text{for any} \ \ t \in [s,T] \ \text{and} \ \ k.
$$
Moreover, if we modify condition \eqref{fwregular} and put the identity on the right hand-side of the inequality, instead of $\theta_f$, we deduce that, for all $t\in [s,T]$ and all $k$, $F_k(t) \in W^{1,2}(\mu,\Omega, d_{\Omega};\mathbb{R}^n)$. In fact, following Hajłasz \cite{Hajlasz1996}, the Sobolev space $W^{1, p}(\mu,\Omega, d_{\Omega};\mathbb{R}^n)$ with $p\in [1,+\infty)$ is defined as follows: $f \in W^{1, p}(\mu,\Omega, d_{\Omega};\mathbb{R}^n)$ if and only if $f \in L^p( \mu,\Omega;\mathbb{R}^n)$ and there exists a positive function $g_{f} \in L^p(\mu,\Omega;\mathbb{R}^n)$ such that
\be\label{L1P}
|f(\omega)-f(\omega')| \leq d_{\Omega}(\omega, \omega')(g_{f}(\omega)+g_{f}(\omega'))
\ee
almost everywhere. The space is equipped with the norm 
\benl
\|f\|_{W^{1, p}(\mu,\Omega, d_{\Omega};\mathbb{R}^n)}:=\|f\|_{L^p( \mu,\Omega;\mathbb{R}^n)}+\inf _{g_{f}}\|g_{f}\|_{L^p( \mu,\Omega;\mathbb{R}^n)},
\eenl
the infimum being taken over all positive $L^p$ functions $g_{f}$  satisfying \eqref{L1P}. 
\end{remark}

 It is worth mentioning that the Banach space $W^{1, 2}(\mu,\Omega, d_{\Omega};\mathbb{R}^n)$ is not necessarily a Hilbert space, unless $\Omega\subset\cR^d,$ in which case, it coincides with the classical Sobolev space $H^{1}(\Omega;\mathbb{R}^n)$ and the well-known compact embedding results hold. However, for an arbitrary measurable metric space $(\Omega,d_{\Omega},\mu)$ this is not the case. A compact embedding result can be obtained if the measure $\mu$ is doubling, {\em i.e.,}  whenever $\mu$ is a positive Borel measure satisfying the condition $$0<\mu(\mathcal{B}_{2r}(\omega))\leq C_{\mu}\mu(\mathcal{B}_{r}(\omega)) < \infty$$
for all $\omega \in \Omega$, $r>0$, and some positive constant $C_{\mu}$. Under this hypothesis, the space $W^{1, p}(\mu,\Omega, d_{\Omega};\mathbb{R}^n)$ is compactly embedded in $L^p( \mu,\Omega;\mathbb{R}^n)$, which gives a version of the Rellich–Kondrachov Theorem for these Sobolev spaces defined in metric spaces \cite{Hajlasz2000, Agnieszka}. Note that condition \eqref{Hyp.mu} is weaker than requiring $\mu$ to be a doubling measure. 
\begin{remark}
    Observe that if the initial datum is such that $\varphi\in C(\Omega;\cR^n)$ (or constant, as supposed in \cite{BettiolKhalil}) the condition \eqref{fwregular} of ${\bf (H_{\mu})}$ in Proposition \ref{Fkprecompact}  can be droppep. In fact, the Lipschitz continuity of $f$ (condition \eqref{fLipschitz}), the continuity of the map $\omega \mapsto f(t,x,u,\omega)$ and the Grönwall's inequality ensure the continuity of $\omega \mapsto F_k(t,\omega)$ for all $t\in [0,T]$ and,  consequently, the sequence $\{F_k(t)\}_k$ satisfies the hypothesis of Theorem \ref{compactL2}. 
\end{remark}

Following Frankowska  \cite[Theorem 2.7]{FRANKOWSKA1990}, one may consider the following alternative and stronger hypothesis:
there exists a compact $K\subset \Vspace$ such that  $$\textbf{(K)} \quad \text{for every} \quad (t,\varphi)\in [s,T]\times \Vspace,\quad F(t,\varphi)\subset K.$$
Then, for all $t \in [s,T],$ we observe that $\{F_k(t)\}_k \subset (T-s)K$ which is evidently relatively compact in $\Vspace.$ Thus, we obtain the result below.

\begin{theorem}[Compactness of the set of trajectories]
\label{compactS}
Under conditions ${\bf (H)},$ ${\bf (H_{\mu})}$ (or, alternatively, ${\bf (K)})$ and ${\bf {(C)}},$ for all $s \in [0,T)$ and $\varphi \in \Vspace),$  the set 
    $
S_{[s,T]}(\varphi)$ is  compact in  $C([s,T]: \Vspace).$
\end{theorem}
\begin{proof} First, we prove that $S_{[s,T]}(\varphi)$ is relatively compact. To achieve this, it is sufficient to verify the hypotheses of the Arzelà-Ascoli Theorem:
        \begin{itemize}
        \item[{\em (1)}] $S_{[s,T]}(\varphi)$ is equicontinuous,
        \\
        \item[{\em (2)}] for all $t\in [s,T],$ 
        \be\label{X(t)}
        \mathcal{X}(t):= \left\{x(t)\right\}_{x \in S_{[s,T]}(\varphi)} \ \text{is relatively compact in} \  \Vspace. 
        \ee
        
    \end{itemize}
\noindent {\em Proof of (1):} Let $t,\tau \in [s,T],$ then for all $x\in S_{[s,T]}(\varphi)$ from Lemma \ref{x.ineq}, it follows that 
$$
\begin{aligned}
    ||x(t) - x(\tau)& ||_{\Vspace} 
    \leq  C\left(\mu(\Omega)^{1/2}+\|\varphi\|_{\Vspace}  \right)|t-\tau |.
\end{aligned}
$$
Consequently, for all $\epsilon >0,$ if $|t-\tau|< \delta := \epsilon \left[C\left(\mu(\Omega)^{1/2}+\|\varphi\|_{\Vspace}  \right)\right]^{-1},$ we have that $||x(t) - x(\tau) ||_{\Vspace}< \epsilon,$ for all $x \in S_{[s,T]}(\varphi).$
\\

\noindent {\em Proof of (2):} If $\textbf{(K)}$ holds, {\em (2)} follows  straightforwardly. Suppose alternatively that ${\bf (H_{\mu})}$ holds and let $t\in [s,T]$ be fixed and $\left\{x_k(t )\right\}_k$ be a sequence in $\mathcal{X}(t).$ For each $k,$ there exists a control $u_k \in \mathcal\Uad[s,T]$ such that 
\be
\label{xk}
    x_k(t )=\varphi+F_{k}(t ),
\ee
where $F_{k}$ was defined in \eqref{Fk}. From Proposition \ref{Fkprecompact},  $\{F_k(t )\}_{k}$ is relatively compact in  $\Vspace,$ thus $\{x_k(t )\}$ is relatively compact in $\Vspace.$ Then,  the Arzelà-Ascoli Theorem ensures the relative compactness of the set $S_{[s,T]}(\varphi)$   in $C([s,T]: \Vspace).$


Now, suppose that $\{x_k\}_k$ is a strongly convergent sequence in $S_{[s,T]}(\varphi)$ such that 
\be
\label{xkstrongx}
x_k \stackrel{s}{\rightarrow} x \quad  \, \text{in} \, \, C([s,T]: \Vspace).
\ee
We prove next that $x \in S_{[s,T]}(\varphi).$ From \eqref{xk}  we have that
$$ 
F_k \stackrel{s}{\rightarrow} F^*\, \, \quad \text{in}\, \, \, C([s,T]: \Vspace), 
$$
for some $F^* \in C([s,T]: \Vspace).$ Then 
\be\label{xlimit}
x(t)=\varphi+F^*(t ), \ \ t\in [s,T].
\ee
On the other hand, by \eqref{fkbounded} we observe that there exists a subsequence (keeping the same index) of $\{F'_k\}=\{\Bar{f}_{k}\} $ satisfying
\be
\label{fweaks}
\Bar{f}_{k} \stackrel{*}{\rightharpoonup} \Bar{f} \, \, \quad \text{in} \, \, L^{\infty}([s,T]:\Vspace),\ee
for some $\Bar{f} \in L^{\infty}([s,T]:\Vspace)$. Moreover, in view of   Theorem \ref{convergences}, 
we can extract a subsequence such that
$$\Bar{f}_{k} \stackrel{w}{\rightharpoonup} \dot{x} \, \, \quad \text{in} \, \, L^1([s,T]:\Vspace).$$
Thus, from the uniqueness of the limits in latter two equations and using \eqref{xlimit}, we deduce that $x$ satisfies 
$$
x(t)=\varphi+\int_{s}^{t}\Bar{f}(\sigma )d\sigma \ \ \quad \text{for } t\in[s,T].
$$
By \eqref{xkstrongx}, for all $\varepsilon>0,$ there exists $k_0$ such that 
$$
x_k(t )\in \mathcal{B}_{\varepsilon}(x(t)) \quad \ \text{for all } \ t\in [s,T], \ \ k\geq k_0,
$$
and 
$$
u_k(t) \in U(t) \subset U(\mathcal{B}_{\varepsilon}(t)) \  \quad \ \text{for all } \ t\in [s,T], \ \ k\geq k_0
$$
Consequently,
\be\label{fkinclu}
\Bar{f}_k(t ) \in f\left(t, \mathcal{B}_{\varepsilon}(x(t)),U(\mathcal{B}_{\varepsilon}(t)) \right) \quad \ \text{a.e.} \ t\in [s,T], \ \ \text{all } \ k\geq k_0.
\ee
The limit in \eqref{fweaks} and the Mazur's Lemma provide us with coefficients $\alpha_{ij}\geq0,$ for which $\sum_{i\geq 1}^{N(i)}\alpha_{ij}=1$ for each $j,$ such that, for some  subsequences $\{\Bar{f}_{ij}\}_{j}$  of $\{\Bar{f}_k\}_{k}$ and  some $p>1,$ the following limit stands
\be
\psi_{j}:=\sum_{i\geq 1}^{N(i)}\alpha_{ij}\Bar{f}_{ij}\stackrel{s}{\longrightarrow} \Bar{f} \quad \text{in} \ L^p([0,T]:\Vspace).
\ee
Then, we deduce that 
$$
\psi_{j}(t )\stackrel{s}{\rightarrow} \Bar{f}(t ) \ \ \text{in} \ \ \Vspace \quad \ \text{a.e.} \ t\in [s,T],
$$
and from \eqref{fkinclu} we have that
$$
\psi_{j}(t ) \in  \operatorname{co} f\left(t, \mathcal{B}_{\varepsilon}(x(t)),U(\mathcal{B}_{\varepsilon}(t)) \right) \ \quad  \text{a.e.} \ \ t\in [s,T].
$$
Thus, for all $\varepsilon>0$ 
$$
\Bar{f}(t ) \in  \overline{\operatorname{co}} f\left(t, \mathcal{B}_{\varepsilon}(x(t)),U(\mathcal{B}_{\varepsilon}(t)) \right) \ \quad  \text{a.e.} \ \ t\in [s,T],
$$
Since the multifunction $U$ is pseudo-continuous, by Proposition \ref{cof} we deduce that 
$$
\Bar{f}(t )\in f\left(t, x(t),U(t) \right) \ \ \text{a.e.} \ \ t\in [s,T].
$$
Then, in view of Corollary \ref{selection},  there exists a measurable control $u: [s,T]\rightarrow \cR^m$, such that
\[
\left\{
\begin{array}{l}
u(t) \in U(t) \quad \text { a.e. } t \in[s, T], \\ \\
\Bar{f}(t )=f(t, x(t), u(t)) \quad \text { a.e. } t \in[s, T].
\end{array}
\right.
\]
This leads us to conclude that $x\in S_{[s,T]}(\varphi)$ and thus the proof follows.
\end{proof}

\begin{remark}
Under hypothesis ${\bf (H_{\mu})}$ or ${\bf (K)},$ the arguments utilized to establish the Arzelà-Ascoli conditions in the proof of Theorem \ref{compactS} above can be adapted to demonstrate the compactness of the operator $\mathcal{I}$ defined in \eqref{OP.INT}.
\end{remark}

\subsection{Lower semi-continuity of the value function}

With the compactness of the set of trajectories in hand, we can solve problem $(P)_{s,\varphi}$ once we have established the lower semi-continuity of the functional ${\mathcal J}$. The result below characterizes the lower semi-continuity of the value function and gives a result of the existence of optimal trajectories. We observe that the value function $V
 $ can be redefined as follows: for all $\left(s, \varphi\right) \in[0, T] \times \Vspace$, 
 \benl\label{vf2}
 V \left(s, \varphi\right):=\inf \Bigg\{ \int_{\Omega}g(x(T,\omega),\omega)d\mu(\omega): x \in S_{\left[s, T\right]}\left(\varphi\right)\Bigg\}.
 \eenl
\begin{theorem}
\label{charV}
If ${\bf (H)},\ ({\bf (H_{\mu})}$ or ${\bf (K)}),$ and $\mathrm{\textbf{(C)}}$  hold true, 
then $V$ is lower semi-continuous and,  
\begin{multline}\label{Vmin}
    V\left(s, \varphi\right)=\min \left\{ \int_{\Omega}g(x(T,\omega),\omega)d\mu(\omega): x \in S_{\left[s, T\right]}\left(\varphi\right)\right\},\\
    \text{for all} \,\, \left(s, \varphi\right) \in[0, T] \times \Vspace. 
\end{multline}

Furthermore, for all $\bar{\varphi} \in \Vspace,$
\be\label{Extpoints}
\begin{aligned} 
\int_{\Omega}g(\bar{\varphi}(\omega),\omega)d\mu(\omega)&=\liminf _{s \rightarrow T^{-}, \varphi \rightarrow \bar{\varphi}} V(s, \varphi),
\\ \\
V(0, \bar{\varphi})&=\liminf _{s \rightarrow 0^{+}, \varphi \rightarrow \bar{\varphi}} V(s, \varphi).
\end{aligned}
\ee
\end{theorem}
\begin{proof}
We note that the compactness of $S_{\left[s, T\right]}\left(\varphi\right)$ implies the compactness of the set $\mathcal{X}(T)$ (defined in \eqref{X(t)}) in $\Vspace$ and, from Theorem \ref{chlsm} (see Appendix), the functional ${\mathcal J}$ is lower semi-continuous. Thus, we obtain \eqref{Vmin}.

In order to prove the semi-continuity of $V,$ we take a sequence $(s_k,\varphi_k)$ converging strongly to $(s,\varphi)$ in  $[0,T]\times \Vspace.$ Let $\bar{x}_k\in S_{[s_k,T]}(\varphi_k)$ be the optimal trajectory associated to $V(s_k,\varphi_k)$ with control $\bar{u}_k\in \Uad[s_k,T]$ for all $k.$ 
We know that the sequence of trajectories $\{\bar{x}_k\}_k$ can be written as 
$$
\bar{x}_k(t )
=\varphi_k+\int_{s_k}^{t}f(\sigma,\bar{x}_k(\sigma ),\bar{u}_k(\sigma)) d\sigma \quad \text{for all} \quad t\in [s_k,T]. 
$$
or
\be\label{barxk}
\bar{x}_k(t)
=\bar{x}_k(s)+\int_{s}^{t}f(\sigma,\bar{x}_k(\sigma ),\bar{u}_k(\sigma)) d\sigma \quad \text{for all} \quad t\in [s_k,T]. 
\ee
where
$$
\bar{x}_k(s)
=\varphi_k+\int_{s_k}^{s}f(\sigma,\bar{x}_k(\sigma ),\bar{u}_k(\sigma)) d\sigma. 
$$
By employing the convergence assumptions mentioned above and that the dynamics $f$
satisfies the inequality \eqref{fkbounded}, we obtain that $$\lim_{k\rightarrow\infty}\int_{s_k}^{s}f(\sigma,\bar{x}_k(\sigma ),\bar{u}_k(\sigma)) d\sigma=0 \quad \text{in} \quad \Vspace,$$  wich implies that $\bar{x}_k(s) \rightarrow \varphi$ in $\Vspace.$
On the other hand, from the proof of Proposition \ref{Fkprecompact} we note that, for all $t\in [s,T]$ the sequence 
{\small$$
 \left\{\int_{s}^{t}f(\sigma,\bar{x}_k(\sigma ),\bar{u}_k(\sigma)) d\sigma\right\}_k
$$}is bounded in $\Vspace$ and satisfies the inequality \eqref{mFk}, then, from Theorem \ref{compactL2}, it is relatively compact in $\Vspace,$ and arguing as in the proof of the Theorem \ref{compactS}, from \eqref{barxk}, we deduce that there exists a subsequence of $\{\bar{x}_k\}_k$ (using the same index) and a process $(x^*,u^*)$ with $x^*\in S_{[s,T]}(\varphi)$ and $u^*\in \Uad[s,T],$ such that $\bar{x}_k \rightarrow x^*$ in $C([s,T]:\Vspace)$ as $k\rightarrow\infty,$ in particular, we observe that 
$\bar{x}_k(T ) \rightarrow x^*(T)$ in $\Vspace$ as $k\rightarrow\infty,$ then, from the semi-continuity of the functional ${\mathcal J}$ we obtain 
\benl
\begin{aligned}
 \liminf_{k\rightarrow \infty}V(s_k,\varphi_k)&=\liminf_{k\rightarrow\infty}\int_{\Omega}g(\bar{x}_k(T,\omega),\omega)d\mu(\omega)
 \\
&\geq\int_{\Omega}g(x^*(T,\omega),\omega)d\mu(\omega)\geq V(s,\varphi).   
\end{aligned}
\eenl

To prove \eqref{Extpoints}, we consider $\varphi$ in $\Vspace$ and  $s_k \rightarrow T^{-},$ $\varphi_k\rightarrow \varphi,$ then the lower semi-continuity of $V$ implies that 
$$
V(T,\varphi)\leq \liminf_{k\rightarrow \infty}V(s_k,\varphi_k).
$$
On the other hand, consider $\{\psi_k\}_k$ in $\Vspace$ and $y_k\in S_{\left[s_k, T\right]}\left(\psi_k\right)$ such that $y_k(T )=\varphi,$ and since there exists a positive constant $C$ (Lemma \ref{x.ineq}) such that 
\benl
    ||\varphi-\psi_k ||_{\Vspace}=||y_k(T ) - y_k(s_k ) ||_{\Vspace}  \leq  C
    |T-s_k |,
\eenl
we conclude that $\psi_k \rightarrow \varphi$ in $\Vspace$ as $k\rightarrow\infty,$ and 
$$
V(s_k,\psi_k)\leq\int_{\Omega}g(y_k(T,\omega),\omega)d\mu(\omega)=\int_{\Omega}g(\varphi(\omega),\omega)d\mu(\omega)=V(T,\varphi),
$$
thus, 
$$
 \liminf_{k\rightarrow \infty}V(s_k,\psi_k)\leq V(T,\varphi).
$$
and the first equality in \eqref{Extpoints} holds.
Now, we denote $x\in S_{\left[0, T\right]}\left(\varphi\right)$ as the optimal trajectory associated to $V(0,\varphi),$ from the Dynamic Programming Principle we know that
$$
V(0,\varphi)=V(s,x(s )) 
 \quad \text{for all} \quad s\in [0,T],$$
 and the second equality follows thanks to the lower semi-continuity of $V.$ 
\end{proof}

\section{Invariance Principles and Proximal Analysis}\label{DP.and.IP}

This section aims at proving that the value function is the unique lower semi-continuous solution of the Hamilton-Jacobi-Bellman equation. We will make use of notions from  proximal normal analysis, the Dynamic Programming Principle 
and invariance principles from \cite{Donchev}.

We begin by recalling basic definitions of invariance principles. Let $X$ be a Hilbert space, $D:[0,\infty)\leadsto X$, $\Gamma:[0,\infty)\times X\leadsto X$  be given multifunctions and $G:=\mathrm{Gr}\,D$.
\begin{definition} Let $D:[0, \infty) \leadsto X$ be closed-valued. We say that $D$ is {\em left absolutely continuous}  when, for every $T,$ every bounded  subset $B\subset X$ and every $\varepsilon>0,$ there exists $\delta>0$ such that, for every sequence $\left\{\left(t_i, s_i\right)\right\}_{i=1}^{\infty}$ of open pairwise disjoint subintervals of $[0, T],$ one has that 
\begin{equation}\label{LAC}
\sum_{i=1}^{\infty}\left(s_i-t_i\right)<\delta \implies \sum_{i=1}^{\infty} \operatorname{ex}\left(D\left(t_i\right) \cap B, D\left(s_i\right)\right)<\varepsilon,
\end{equation}
where $\operatorname{ex}(U,V) := \sup_{u\in U}\operatorname{dist}(u,V)$ and $\operatorname{dist}(u,V):=\inf_{v\in V}|u-v|$.
For right absolute continuity and absolute continuity, one replaces the expression in the sum of the right hand-side of \eqref{LAC} by 
$\operatorname{ex}\left(D\left(s_i\right) \cap B,D\left(t_i\right)\right)$ respectively, by $\operatorname{ex}_{B}\left(D\left(t_i\right), D\left(s_i\right)\right),$ where
$$
\operatorname{ex}_{B}\left(D\left(t_i\right), D\left(s_i\right)\right):=\max \left\{\operatorname{ex}\left(D\left(t_i\right) \cap B, D\left(s_i\right)\right), \operatorname{ex}\left(D\left(s_i\right) \cap B, D\left(t_i\right)\right)\right\}.
$$
\end{definition}

\begin{definition}[Weak invariance]
\label{def:weak}
The graph $G$ of $D$ is {\em weakly invariant} w.r.t. the set-valued dynamics $\dot{x}\in \Gamma(t,x)$ (and we write $(\Gamma, G)$ is weakly invariant) if, for any initial condition $x_0\in D(t_0),$ there exists $T>t_0,$ such that the Cauchy problem 
\begin{equation}\label{inv_dyn}
\left\{
    \begin{aligned}
    &\dot{x}(t)\in \Gamma(t,x(t))\quad t\in[t_0,T],\\
    &x(t_0)=x_0,
    \end{aligned}
\right.
\end{equation}
admits a solution $x(t)\in D(t)$ for all $t\in[t_0,T)$.
\end{definition}
\begin{definition}[Strong invariance]
\label{def:strong}
The graph $G$ of $D$ is {\em strongly invariant} w.r.t. the set-valued dynamics $\dot{x}\in \Gamma(t,x)$ (and we write $(\Gamma, G)$ is strongly invariant) if, for any initial condition $x_0\in D(t_0)$ and any $T>t_0$, every solution of the Cauchy problem \eqref{inv_dyn} satisfies the condition $x(t)\in D(t),$ for all $t\in[t_0,T]$.
 \end{definition}
\begin{definition}[Proximal normal \cite{Clarke1998}]
    Let $K \subset X$ be a closed set. A {\em proximal normal} to $K$ at a point $x \in K$ is a vector $\xi \in X$ such that there exists $\lambda>0$ satisfying $$\left\langle\xi, x^{\prime}-x\right\rangle_{X} \leq \lambda||x^{\prime}-x||^2_X \quad  \text{for all } x^{\prime} \in K.$$ The set of all such vectors is a cone denoted by $N_K^{P}(x)$ and called {\em proximal normal cone} to $K$ at $x$. 
\end{definition}
\begin{definition}[Proximal subgradient and subdifferentials \cite{Clarke1998}]
    A vector $\xi \in X$ is called a {\em proximal subgradient} (shortly, {\em $P$-subgradient}) of an extended -valued lower semi-continuous function $V:X\rightarrow \cR \cup \{ +\infty\}$ at $x \in \operatorname{dom}V$ whenever
$$
(\xi,-1) \in N_{\text {epi } V}^P(x, V(x)) .
$$
The (possibly empty) set of all such $\xi$ is denoted by $\partial_P V(x)$, and is referred to as {\em proximal subdifferential} or {\em $P$-subdifferential}. If $V$ is an upper semi-continuous function, one sets $\partial^{P}V(x):=-\partial_{P}(-V(x)).$
\end{definition}

Now, we consider the  Hilbert space $X=\Vspace$ and introduce the notion of {\em proximal solution} for the HJB equation \eqref{HJB} which, in the continuous case, is equivalent to the definition of viscosity solution \cite{ClarkeLedyaev1994} introduced by Crandall \& P.-L. Lions \cite{CraLio83}.
\begin{remark}
    In view of Fillipov's Lemma \ref{Filippov}, the Hamiltonian $$H:[0,T]\times \Vspace \times \Vspace^* \rightarrow \mathbb{R}$$ defined in \eqref{Hamiltonian} can be rewritten as follows
   \be
H(t,\varphi,p)=\min_{v\in F(t,\varphi)} \langle v, p\rangle. 
   \ee
\end{remark}
\begin{definition}[Proximal solution \cite{ClarkeLedyaev1994}]
\label{prox.sol}
Let $A\subset [0,T]\times\Vspace$ be an open set. A lower semi-continuous function $W\colon A\rightarrow \mathbb{R}$ is a {\em proximal supersolution} of \eqref{HJB} if
$$
-[\xi_t +\min_{v\in F(t,x)} \langle v, \xi_{\varphi}\rangle ] \geq 0,
$$
for every $(t,\varphi)\in A$ and $(\xi_t,\xi_\varphi)\in \partial_{P}W(t,\varphi)$.

An upper semi-continuous function $W\colon A\rightarrow \mathbb{R}$ is a {\em proximal subsolution} of \eqref{HJB} if
$$
-[\xi_t +\min_{v\in F(t,x)} \langle v, \xi_{\varphi}\rangle ] \leq 0,
$$
for every $(t,\varphi)\in A$ and $(\xi_t,\xi_\varphi)\in \partial^P W(t,\varphi)$.

A  continuous function $W\colon A\rightarrow \mathbb{R}$ is a {\em proximal solution} if it is both proximal supersolution and subsolution.
\end{definition}

Let us now recall some results of weak and strong invariance, and Hamiltonian characterizations (see, e.g. T. Donchev \cite{Donchev}).

\begin{theorem}[Characterization of weak invariance \cite{Donchev}]
\label{w.i.-char}
Assume that $\Gamma$ is almost upper semi-continuous, with compact and convex values. Suppose that there exists $c\in L^1(0,\infty)$ such that $|\Gamma(t,x)|\leq c(t)$ a.e. $t\in[0,\infty)$. Then the pair $(\Gamma, G)$ is weakly invariant if and only if the following two conditions hold true:
\begin{itemize}
\item[i)] the set-valued map $t\leadsto D(t)$ is left absolutely continuous;
\item[ii)] there exists a full measure set $I$ in $[0,\infty)$ such that, for every $(t,x)\in G$ with $t\in I$, one has
$$\xi_0+\min_{v\in \Gamma(t,x)} \langle v, \xi \rangle \leq 0\quad \text{for all }
\,(\xi_0,\xi)\in N^P_G(t,x).$$
\end{itemize}
\end{theorem}

\begin{theorem}[Characterization of strong invariance \cite{Donchev}]
\label{s.i.-char}
Assume that $\Gamma$ is almost lower semi-continuous and such that $x\leadsto \Gamma(t,x)$ is Lipschitz continuous, a.e. $t\in[0,\infty)$.  Then the pair $(\Gamma, G)$ is strongly invariant if and only if  there exists a full measure set $I$ in $[0,\infty)$ such that, for every $(t,x)\in G$ with $t\in I$, one has
$$\xi_0+\max_{v\in \Gamma(t,x)}  \langle v , \xi \rangle\leq 0\quad \text{for all} \;\; (\xi_0,\xi)\in N^P_G(t,x).$$
\end{theorem}

\subsection{Invariance properties and the HJB equation}

Now, we will demonstrate the invariance properties of an autonomous system involving the multifunction
$F$ (see \eqref{Fdef}). These results will show that the value function $V$ defined in \eqref{vf}  is the unique lower semi-continuous solution of the HJB equation. This is done by using the Dynamical Programming Principle given in Theorem \ref{DPP}. 

Other properties of $F$ can be derived from the assumptions ${\bf (H)},$ in fact, one has that

1. for all $t\in [0,T]$ fixed, $F(t,\cdot)$ is continuous; and 

2. $F(t,\varphi)$ is $\mathcal{L}\otimes\mathcal{B}$- measurable. 

\noindent Consequently, from a Scorza Dragoni type result \cite{ArtsteinPrikry1987}, there exists a full measure interval $I\subset [0,T]$ such that $F(t,\varphi)$ is continuous in $I\times\Vspace.$ Observe that these properties guarantee that the multifunctions $$\{1\}\times F(t,\varphi) \times \{0\} \quad \text{and} \quad \{1\}\times -F(T-t,\varphi) \times \{0\}$$  satisfy the regularity hypotheses of Theorems \ref{w.i.-char} and \ref{s.i.-char}, respectively. 

Let $V$ be the value function of problem $(P)_{s,\varphi},$ and  
let us define the set
$$P(t):=\left\{(\varphi,\alpha)\in \Vspace\times \mathbb{R}: V(t,\varphi)\leq \alpha \right\},$$
 and consider the associated set-valued function $t\leadsto P(t).$
For each $t\in[0,T]$, such a set $P(t)$ can be viewed as the epigraph of the function $\Vspace \ni \varphi\mapsto V(t,\varphi),$ 
 and $\mathrm{Gr}\,P$ coincides with $$\mathrm{epi}\,V:=\left\{(t,\varphi,\alpha)\in \Vspace\times \mathbb{R}: V(t,\varphi)\leq \alpha \right\} ,$$ the epigraph of the value function, which is a closed set due to the lower semi-continuity of $V.$

\begin{proposition}
If hypotheses {\bf (H)}, ${\bf (H_{\mu})}$ (or ${\bf (K)}$) and $\mathrm{\textbf{(C)}}$  hold true, then the set-valued map $$t\leadsto P(t) $$ is absolutely continuous. \end{proposition}
\begin{proof} 
Fix any $0\leq s \leq \tau \leq T,$ let $B\times I\subset \Vspace \times \cR$ be a bounded set and $(\varphi,\alpha)\in P(s)\cap  B\times I,$ consider $\bar{x}\in S[s,T](\varphi)$ the optimal solution, then $V(t,\bar{x}(t))=\int_{\Omega}g(\bar{x}(T,\omega),\omega)d\mu(\omega)$ for all $t\in[s,T].$ This implies that $V(\tau,\bar{x}(\tau))\leq \alpha,$ and $(\bar{x}(\tau),\alpha)\in P(\tau).$ Thus, from Lemma \ref{x.ineq} we have that 
\benl
\begin{aligned}
    \operatorname{dist}((\varphi,\alpha),P(\tau))&\leq \|\varphi-\bar{x}(\tau)\|_{\Vspace}\\ 
    &\leq C\left(\mu(\Omega)^{1/2}+\operatorname{diam}(B)  \right)|\tau-s |,
\end{aligned}
\eenl
which leads to$$\operatorname{ex}\left(P\left(\tau\right)\cap  B\times I, P\left(s\right)\right)\leq C\left(\mu(\Omega)^{1/2}+\operatorname{diam}(B)  \right)|\tau-s |.$$
On the other hand, let $(\varphi,\alpha)\in P(\tau)\cap B\times I $ and consider $x\in S_{[s,\tau]}$ such that $x(\tau)=\varphi.$ Then $V(s,x(s))\leq V(\tau,\varphi) \leq \alpha,$ and $(x(s),\alpha)\in P(s).$ Thus 
\benl
\begin{aligned}
    \operatorname{dist}((\varphi,\alpha),P(\tau))
    &\leq \|x(\tau)-x(s)\|_{\Vspace}
\\ 
&\leq C\left(\mu(\Omega)^{1/2}+\operatorname{diam}(B)  \right)|\tau-s |,
\end{aligned}
\eenl
and
$$\operatorname{ex}\left(P\left(s\right)\cap  B\times I, P\left(\tau\right)\right)\leq C\left(\mu(\Omega)^{1/2}+\operatorname{diam}(B)  \right)|\tau-s |.$$


\end{proof}


\begin{proposition}[Invariance results on $V$]
\label{dpp-inv}
Let us assume {\bf (H)}, ${\bf (H_{\mu})}$ (or ${\bf (K)}$) and the convexity assumption  $\mathrm{\textbf{(C)}}.$   Then:
\begin{itemize}
\item[$i)$] the pair $(\{1\}\times F\times\{0\}, {\rm Gr}\, P)$   
is weakly invariant;
\item[$ii)$] the pair $(\{1\}\times\tilde F \times\{0\}, {\rm Gr}\, \tilde P)$ 
is strongly invariant, where $\tilde F(t, \varphi):=-F(T-t,\varphi)$ and $\tilde P(t):=P(T-t)$ for every $(t,\varphi)\in [0,T]\times \Vspace$.
\end{itemize}
\end{proposition}
\begin{proof} $i).$ Take any $(t_0, \varphi_0, \alpha_0)\in [0,T]\times \Vspace\times\mathbb{R}$ such that  $(\varphi_0, \alpha_0)\in P(t_0)$. Take the optimal pair $(\bar x, \bar u)$ of $(P)_{(t_0, \varphi_0)}$.
It follows from the Dynamic Programming Principle that $V(t, \bar x(t))=V(t_0, \varphi_0)\leq \alpha_0$ for every $t\in [t_0, T]$, implying that $(\bar x(t), \alpha_0)\in P(t)$ for all $t\in [t_0, T]$.

$ii).$ Take any $(t_0, \varphi_0, \alpha_0)\in [0,T]\times \Vspace\times\mathbb{R}$ such that $(\varphi_0, \alpha_0)\in \tilde P(t_0)$. Consider any solution $\tilde{x}(\cdot)$ of the problem
\[
\left\{ \begin{array}{ll}
\dot{x}(t)\in \tilde F(t, x(t)), \qquad \mathrm{a.e.}\; t\in [t_0, T], \\
x(t_0)=\varphi_0.
\end{array} \right . 
\]
One needs to show that
$$(\tilde x(t), \alpha_0)\in \tilde P(t)= \left\{(\varphi,\alpha)\in \Vspace\times \mathbb{R}:\quad \tilde{V}(t,\varphi)\leq \alpha \right\},$$
where $\tilde V(t,x):= V(T-t,x)$.
Let us define the curve $y(s):=\tilde{x}(T-s)$, $s\in [0, T-t_0]$ (namely, $s=T-t$), and observe that
$$\frac{d}{ds}y(s)=\frac{d}{ds}\tilde{x}(T-s)\in -\tilde{F}(T-s, \tilde{x}(T-s))=F(s,y(s)).$$
In view of the Dynamic Programming Principle, one has that
$$V(s,y(s))\leq V(T-t_0, y(T-t_0))=\tilde{V}(t_0,\tilde{x}(t_0))\leq \alpha_0,\qquad \forall\; s\in [0,T-t_0].$$
Since $s=T-t$, the left hand side of the previous inequality reads as
$$V(s,y(s))=V(s, \tilde x(T-s))=V(T-t, \tilde{x}(t))=\tilde{V}(t, \tilde{x}(t)),$$
providing the inequality
$$\tilde{V}(t, \tilde{x}(t))\leq \tilde{V}(t_0,\tilde{x}(t_0))\leq \alpha_0,\qquad \forall\; t\in [t_0,T],$$
which gives the desired condition.
\end{proof}

Let us recall the definition of the functional $\mathcal{J}$ given in \eqref{calJ}.

\begin{theorem}[Comparison results for $V$]
\label{comparison}
Take a lower semi-continuous function $\theta:[0,T]\times \Vspace\rightarrow (-\infty, \infty]$  such that 
$$ \theta(T, \varphi)=\mathcal{J}(\varphi) \qquad \text{for all }\; \varphi \in \Vspace.$$
Define
$$P_{\theta}(t):=\left\{(\varphi, \alpha)\in \Vspace\times \mathbb{R}:\quad \theta(t,\varphi)\leq \alpha  \right\}.$$
Then
\begin{itemize}
\item[$i)$] if the pair $(\{1\}\times F\times \{ 0\}, \mathrm{Gr}\, P_{\theta})$ is weakly invariant, then $V(t,\varphi)\leq \theta(t,\varphi)$ for all $(t,\varphi)\in [0,T)\times \Vspace$.
\item[$ii)$] if the pair $(\{1\}\times  \tilde F \times\{0\},\mathrm{Gr}\, \tilde{P}_\theta)$ is strongly invariant, then $V(t,\varphi)\geq \theta(t,\varphi)$ for all $(t,\varphi)\in [0,T]\times \Vspace$ where $\tilde F(t, \varphi):=-F(T-t,\varphi)$ and $\tilde{P}_\theta(t):=P_\theta(T-t)$ for every $(t,\varphi)\in [0,T]\times \Vspace.$
\end{itemize}
\end{theorem}
\begin{proof} $i)$
Since the pair $(\{1\}\times F\times \{ 0\},\mathrm{Gr}P_{\theta})$ is weakly invariant then, for every $(t_0, \varphi_0, \alpha)\in [0,T)\times \Vspace\times \mathbb{R}$, there exists a trajectory $x$ with $\dot{x}(t,\cdot)\in F(t,x(t,\cdot))$ for a.e. $t\in [t_0, T)$, such that $x(t_0)=\varphi_0$ and $(x(t), \alpha)\in P_\theta(t)$ for all $t\in [t_0, T)$. Choosing $\alpha=\theta(t_0,\varphi_0)$, the weak invariance property implies 
$$ \theta(t,x(t))\leq\theta(t_0,\varphi_0) \qquad \text{for all } \, t\in[t_0, T).
$$
In view of the lower semi-continuity of $\theta$, letting $t\rightarrow T^-$, one obtains the inequality
$$
\mathcal{J}(x(T))\leq \liminf_{t\rightarrow T^-} \theta(t, x(t))\leq\theta(t_0,\varphi_0)$$
and, taking the infimum over all the trajectories, one gets
$$V(t_0,\varphi_0)\leq \theta(t_0,\varphi_0).$$

$ii)$  Since the pair $(\{1\}\times \tilde F \times\{0\}, \mathrm{Gr}\tilde{P}_\theta)$ is strongly invariant, this implies that, for every $(t_0, \varphi_0, \alpha)$ such that $(\varphi_0, \alpha)\in \tilde P_\theta(t_0)$, every solution of the Cauchy problem 
\be\label{Cau.-F}
\begin{cases}
    \dot{x}(t)\in \tilde F(t,x(t)),\\
    x(t_0)=\varphi_0, 
\end{cases}
\ee
satisfies the condition
$$
\theta(T-t, x(t))\leq \alpha \qquad \text{for all }\, t\in[t_0, T].
$$
Let us fix $t_0\in \mathbb{R}$, $\varphi_0\in \Vspace$ and $\alpha=\theta(T, \varphi_0)=\mathcal{J}(\varphi_0)$. Then the previous inequality reads as
\begin{equation}\label{theta_good}
\theta(T-t, x(t))\leq \mathcal{J}(\varphi_0) \qquad \text{for all } t\in[t_0, T]
\end{equation}
for any solution of the Cauchy problem \eqref{Cau.-F}. Let us define $y(s):=x(T-s+t_0)$ and observe that 
$$
\frac{d}{ds}y(s)= \frac{d}{ds}x(T-s)\in - \tilde{F}(s, x(T-s))= F(s, y(s)) \qquad \text{a.e. }\, s\in [t_0, T],
$$
and final condition $y(T)=x(t_0)=\varphi_0$. Reformulating the relation \eqref{theta_good} with respect to the new time variable $s=T-t+t_0$, one obtains
$$
\theta(s, y(s))\leq {\mathcal J}(y(T)) \qquad \text{for all } s\in [t_0,T].$$
Finally,  taking the infimum over the set of trajectories in latter equation, one gets
$$\theta(s, y(s))\leq V(s, y(s)),$$
which is the desired relation when $s=t_0$. This concludes the proof.
\end{proof}

$$
\operatorname{meas}\{t \in [0,T] | \,  \exists \, \varphi\in L^2, N^P_{\mathrm{Gr} P} (t,\varphi,V(t,\varphi))\subset \mathbb{R}\times \Vspace \times \{0\} \}=0
$$

\begin{theorem}\label{teo.prox.sol} Let us assume that ${\bf (H)},$ ${\bf (H_{\mu})}$ (or ${\bf (K)}),$ and  $\mathrm{\textbf{(C)}}$ hold true. Suppose that 
\be\label{q=0}
\begin{aligned}
    \operatorname{meas}\{t &\in [0,T] | \,  \exists \, \varphi\in \Vspace, \\ &\ \{ 0\}\neq N^P_{\mathrm{epi}\,V} (t,\varphi,V(t,\varphi))\subset \mathbb{R}\times \Vspace \times \{0\} \}=0.
\end{aligned}
\ee
Then the value function $V$ of problem $(P)_{s,\varphi}$ is the unique lower semi-continuous, bounded from below function  such that
 there exists a set $I\subseteq [0,T)$ of full measure for which, for every $(t,\varphi, \alpha)\in \mathrm{epi}V \cap (I\times \Vspace\times \mathbb{R})$, one has
\be\label{HJBproximal}
\begin{array}{ll}
\displaystyle \xi_t+ \min_{v\in F(t, \varphi)} \left< v, \xi_{\varphi} \right>=0 \qquad \text{for all }\; (\xi_t, \xi_{\varphi})\in \partial_{P}V(t,\varphi),\\[3mm]
V(T,\varphi)=\mathcal{J}(\varphi).
\end{array}
\ee
\end{theorem}
\begin{proof}
 In view of Theorem \ref{charV}, Proposition \ref{dpp-inv} and Theorem \ref{w.i.-char},  the value function is a lower semi-continuous, bounded below  function and there exists a full measure interval $I_1\subseteq [0,T)$ such that, for every $(t,\varphi,\alpha)\in \mathrm{epi}V \cap \left(I_1\times \Vspace\times \mathbb{R}\right)$, one has
\begin{equation}\label{ineq-wi-proof}
\xi_0+ \min_{v\in F(t, \varphi)} \left< v, \xi \right>\leq 0 \qquad \forall\; (\xi_0, \xi, -q)\in N^P_{\mathrm{Gr} P} (t,\varphi,\alpha).\\[2mm]
\end{equation}
Set $s=T-t$ and  observe that, in view of the definition of $\mathrm{Gr}{P}$, the condition $(T-s,\varphi,\alpha)\in \mathrm{Gr} P$ implies $(s,\varphi,\alpha)\in \mathrm{Gr} \tilde{P}$. Furthermore, an easy computation of the proximal normal cone $N^P_{\mathrm{Gr}P}$ shows that, for every $(T-s,\varphi,\alpha)\in \mathrm{Gr} P$,
$$
\text{if} \;\;  (\xi_0, \xi, -q)\in N^P_{\mathrm{Gr} P} (T-s,\varphi,\alpha),\;\; \text{then}\;\; (-\xi_0, \xi, -q)\in N^P_{\mathrm{Gr}\tilde{P}} (s,\varphi,\alpha).
$$
On the other hand, in view of Proposition \ref{dpp-inv} and Theorem \ref{s.i.-char},   there exists a full measure interval $I_2\subseteq [0,T]$ such that 
the following  conditions hold 
\benl
\begin{array}{ll}
-\xi_0+ \displaystyle\max_{v\in \tilde{F}(s, \varphi)} \left< v, \xi \right>\leq 0 \qquad \text{for all }\; (-\xi_0, \xi, -q)\in N^P_{\mathrm{Gr} \tilde {P}} (s,\varphi,\alpha),
\\
[4mm]
V(T,\varphi)=\mathcal{J}(\varphi),
\end{array}
\eenl
one can rewrite the conditions above as
\begin{equation}\label{ineq-si-proof}
\begin{array}{ll}
-\xi_0+ \displaystyle\max_{v\in -F(t, \varphi)} \left< v, \xi \right>\leq 0 \qquad \text{for all }\; (-\xi_0, \xi, -q)\in N^P_{\mathrm{Gr} \tilde {P}} (T-t,\varphi,\alpha),
\\
[4mm]
V(T,\varphi)=\mathcal{J}(\varphi).
\end{array}
\end{equation}
In particular the relations \eqref{ineq-wi-proof} and \eqref{ineq-si-proof} imply that there exists a full measure interval $I_3:=I_1\bigcap I_2\subseteq [0,T]$ such that, for every $(t,\varphi,\alpha)\in \mathrm{epi}V \cap \left(I_3\times \Vspace\times \mathbb{R}\right)$, the following conditions hold:
\begin{equation}\label{characterization}
\begin{array}{ll}
\displaystyle \xi_0+ \min_{v\in F(t, \varphi)} \left< v, \xi \right>=0 \qquad \text{for all }\; (\xi_0, \xi, -q)\in N^{P}_{\mathrm{epi} V}(t,\varphi,\alpha),\\[3mm]
V(T,\varphi)=\mathcal{J}(\varphi).
\end{array}
\end{equation}
From the properties of the proximal cone $N^{P}_{\mathrm{epi} V}$ (see e.g. \cite{Clarke1998}), we have that $q \geq 0:$ $\alpha=V(t,\varphi)$ if $q>0,$ and $q=0$ if $\alpha> V(t,\varphi).$ In this last case, {it follows that  
\be\label{q=0b} (\xi_0,\xi,0)\in N^{P}_{\mathrm{epi} V}(t,\varphi,V(t,\varphi) ),
\ee 
and from \eqref{q=0}, we deduce that there exists a full-measure interval $I_4\subset[0,T]$ such that \eqref{q=0b} does not hold. Thus, we set $I:=I_3 \cap I_4$ and consider the remaining case $q>0.$ Define} $\xi_t:= \xi_0 /q$ and  $\xi_{\varphi}:=\xi /q.$ The convexity of the proximal cone ensures that 
\be\label{Pro.N}
(\xi_t,\xi_{\varphi},-1)\in N^{P}_{\mathrm{epi} V}(t,\varphi,V(t,\varphi) )
\ee
and since the Hamiltonian $\min_{v\in F(t, \varphi)} \left< v, \xi \right>$ is positively homogeneous with respect $\xi,$ the relations \eqref{characterization} and \eqref{Pro.N} imply the conditons \eqref{HJBproximal}.  

Now, suppose that $\theta: [0,T]\times \Vspace\rightarrow \mathbb{R}$  is any other  lower semi-continuous, bounded from below  function satisfying the relation \eqref{HJBproximal} 
and  such that the set-valued map
$$t\leadsto P_{\theta}(t):=\left\{(\varphi, \alpha)\in \Vspace\times \mathbb{R}:\quad \theta(t,\varphi)\leq \alpha  \right\}$$
is absolutely continuous. Then, in view of Theorems \ref{w.i.-char}, \ref{s.i.-char} and of Proposition \ref{comparison}, $\theta\equiv V$. This concludes the proof.
\end{proof}




\appendix

\if{
\begin{theorem}[Dynamic Programming Principle]
 Under the  hypothesis {\bf (H)}, the following conditions hold true:
\begin{itemize}
\item[i)] for every $\varphi \in \Vspace$, for every and $s_{1},s_{2}\in [0,T]$, with $s_{1}\leq s_{2}$, one has
\begin{equation*}
V(s_{1},\varphi) \leq V(s_{2},x_{\varphi,u}(s_{2};\,\cdot)),
\end{equation*}
for any measurable control $u\in \Uad[s_1,s_2],$ where $x_{\varphi,u}$ is the trajectory with initial condition $x_{\varphi,u}(s_1,\cdot)=\varphi(\cdot)$ and associated to control $u\in \Uad[s_1,s_2];$
\item[ii)] for every $\varphi \in \Vspace$  and $s_{1},s_{2}\in [0,T]$, with $s_{1}\leq s_{2}$, one has
\begin{equation*}
V(s_{1},\varphi) = \inf_{u\in \Uad[s_1,T]} V(s_{2},x_{\varphi,u}(s_2,\cdot)).
\end{equation*}
\end{itemize}
\end{theorem}
}\fi

\section{Proof of Theorem \ref{DPP}: the Dynamic Programming Principle}
\label{ProofDPP}

\begin{proof} 
i). Suppose by contradiction that the assertion in i) of Theorem \ref{DPP} does not hold. Then there exist $s_{1}, s_{2}\in [0,T]$, $s_{1}<s_{2}$,  a control  $u_1 \in \Uad[s_1,s_2]$ and $\epsilon>0$ such that
\begin{equation}
V(s_{1},\varphi) > V(s_{2},x_{\varphi,u_1}(s_{2};\cdot)) + \epsilon
\end{equation}
On the other hand, by  characterization of the infimum,  there exists a control $u_2\in \Uad[s_2,T]$ such that
$$
V(s_{2},x_{\varphi,u_1}(s_{2};\cdot))>J_{[s_2,T]}(x_{\varphi,u_1}(s_2,\cdot),u_2) -\frac{\eps}{2}.$$
Consider the control $u\in \Uad[s_1,T]$ obtained from concatenating $u_1$ and $u_2:$
\begin{equation*}
u(s):=\left \{ 
\begin{array}{ll}
\ds u_1(s) \qquad s\in [s_{1}, s_{2}),\\
\ds u_2(s) \qquad s\in [s_{2},T].
\end{array} \right .
\end{equation*}
It easily follows from the definition of value function that
\begin{multline}
 V(s_{1},\varphi) \leq J_{[s_1,T]}(\varphi,u) = J_{[s_2,T]}(x_{\varphi,u_1}(s_2,\cdot),u_2) \\
< V(s_{2},x_{\varphi,u_1}(s_{2};\cdot))+ \frac{\eps}{2} < V(s_{1},\varphi)- \frac{\eps}{2},
\end{multline}
which leads a contradiction. This completes the proof of i).

ii) It only remains to prove that 
$$
V(s_{1},\varphi) \geq \inf_{u\in \Uad[s_1,T]} V(s_{2},x_{\varphi,u}(s_2,\cdot)).
$$
Let $\eps>0$ and $u_{{\eps}}\in\Uad[s_1,T]$ be a control such that
$$
J_{[s_1,T]}(\varphi,u_{{\eps}}) \leq V(s_1,\varphi) + \eps.
$$
Observe that
$$
V(s_2,x_{\varphi,u_{{\eps}}}(s_2,\cdot)) \leq J_{[s_2,T]}(x_{\varphi,u_{{\eps}}}(s_2,\cdot),u_{{\eps}}).
$$
Then
\benl
\begin{aligned}
\ds \inf_{u\in \Uad[s_1,T]} V(s_{2},x_{\varphi,u}(s_2,\cdot)) & \leq V(s_{2},x_{\varphi,u_{{\eps}}}(s_2,\cdot))\\
 & \ds \leq J_{[s_2,T]}(x_{\varphi,u_{{\eps}}}(s_2,\cdot),u_{{\eps}}) \\
 &= J_{[s_1,T]}(\varphi,u_{{\eps}}) \leq V(s_1,\varphi) + \eps.
 \end{aligned}
\eenl
This yields the desired inequality and completes the proof.
\end{proof}

\section{Measurable Selections and integral functionals}
\label{AppMeasurable}


Let $\mathbf{T}$ and $X$ be two topological spaces. Let $2^X$ be the set of all nonempty subsets of $X$. We call any map $\Gamma: \mathbf{T} \rightarrow 2^X$ a multifunction, we will denote it as $\Gamma: \mathbf{T} \leadsto X.$

\begin{definition}[Continuity of set-valued maps]
     A multifunction $\Gamma:\mathbf{T} \leadsto X$ is said to be {\em upper semi-continuous} if, for any closed subset $F \subset X$, the set$$\Gamma^{-1}(F) \equiv\{t \in \mathbf{T} \mid \Gamma(t) \bigcap F \neq \phi\}$$is closed in $\mathbf{T}.$ $\Gamma$ is said to be {\em lower semi-continuous} if for any open subset $U \subset X, \Gamma^{-1}(U)$ is open in $\mathbf{T}$. Finally, $\Gamma$ is said to be {\em continuous} if it is both upper and lower semi-continuous.
     \end{definition}
When $\Gamma$ is single-valued, the above three kinds of continuities are equivalent.

\begin{definition}[Polish space]
    A topological space $X$ is called a {\em Polish space} if it has a countable base, is metrizable, and the space is complete under a metric compatible with the topology of the space.
\end{definition}

\begin{definition}[Souslin space]
    A topological space $X$ is called a {\em Souslin space} if it is Hausdorff and there exists a Polish space $P$ and a continuous map $g: P\rightarrow X,$ such that $$g(P) = X.$$ A subset $A \subset X$ is said to be {\em Souslinian} if as a subspace, A is a Souslin space. Sometimes, we also call such an A a {\em Souslin set}.
\end{definition}

\begin{definition}
    A multifunction $\Gamma: {\bf T} \leadsto X$ is said to be {\em Lebesgue (resp. Borel, Souslin) measurable} if for any closed set $F\subset X,$ the set $\Gamma^{- 1}(F)$ is Lebesgue (resp. Borel, Souslin) set in {\bf T}.
\end{definition}
We will suppose that $\mathbf{T}=[0,T]$ and $X=\cR^m$  with the usual topology on $\cR$ and $\cR^m$ respectively.
\begin{definition}[Pseudocontinuity of set-valued functions]
\label{pseudocontinuous}
 A multifunction $U: [0,T] \leadsto \cR^m$ is said to be {\em pseudo-continuous} at $t \in [0,T]$ if $$\bigcap_{\varepsilon>0} \overline{U\left(\mathcal{B}_{\varepsilon}(t)\right)}=U(t) .$$We say that $U$ is pseudo-continuous on $[0,T]$ if it is pseudo-continuous at each point $t \in [0,T]$.
\end{definition}

\begin{proposition}
\label{multifunctionU}
     Let ${\bf U} \subset \cR^m $ be a compact metric space and $U:[0,T]\leadsto {\bf U} $ be a multifunction taking closed set values. Then $U$ is upper semi-continuous if and only if it is pseudo-continuous.
\end{proposition}

\begin{theorem}\label{Filippov}[Filippov's Lemma \cite[P. 102]{liYong}]
Let $U: [0,T] \leadsto \cR^m$ be measurable taking closed set values and $Y$  a Polish space. Let $f: [0,T] \times \cR^m \rightarrow Y$ be Souslin measurable and for each $\bar{u} \in \cR^m$, $f(\cdot, \bar{u})$ is measurable; for almost all $t \in [0,T],$ $ f(t, \cdot)$ is continuous. Let $y$ : $[0,T] \rightarrow Y$ be Lebesgue measurable satisfying
$$
y(t) \in f(t,U(t)), \quad \text { a.e. } t \in [0,T] \text {. }
$$
Then there exists a measurable function $u: [0,T] \rightarrow \cR^m$, such that
$$
\left\{\begin{array}{l}
u(t) \in U(t) \quad \text { a.e. } t \in [0,T], \\ \\
y(t)=f(t, u(t)) \quad \text { a.e. } t \in [0,T] .
\end{array}\right.
$$
\end{theorem}

\begin{corollary}\label{selection}
If $f(t,\cdot,\cdot)$ is measurable for almost all $t\in [s,T]$ and 
\benl
    \dot{x}(t) \in f(t,x(t),U(t))\ \ \ \text{a.e.}
    \ t \in [s,T],
\eenl
then, there exists a measurable function $u:[0,T]\rightarrow \cR^m $, such that
$$
\left\{\begin{array}{l}
u(t) \in U(t), \quad \text { a.e. } t \in[s, T], \\ \\
\dot{x}(t)=f(t, x(t), u(t)), \quad \text { a.e. } t \in[s, T] .
\end{array}\right.
$$
\end{corollary}
\begin{proof}
    It follows from Theorem \ref{Filippov}.
\end{proof}

\begin{proposition}\cite{liYong}\label{cof}
   Let $U:[0,T]\leadsto \cR^m $ be a pseudo-continuous multifunction and  $f(t,\varphi,u)$ be uniformly continuous on $(\varphi,u)\in \Vspace\times\cR^m$ for any $t\in [0,T].$ Then $f(t, \varphi, U(t))$ is closed and convex if and only if, for almost all $t \in[0, \infty)$, the set $f(t, \varphi, U(t))$ satisfies the following:
\be \label{clo.con}
\bigcap_{\delta>0} \overline{\operatorname{co}} f\left(t, \mathcal{B}_\delta(\varphi), U\left(\mathcal{B}_\delta(t)\right) \right)=f(t, \varphi, U(t) ),
\ee
where $\overline{\operatorname{co}}\, S$
is the smallest convex and closed set containing $S,$ this is, $\overline{\operatorname{co}} \, S$ denotes the convex hull of $S.$
\end{proposition} 
\begin{remark}\label{F.compact}
   If \eqref{clo.con} holds, from the compactness of $U(t)$ for all $t\in[0,T]$ and the continuity of $f$ with respect to $u,$ it follows that $f(t, \varphi, U(t) )$ is compact in $\Vspace,$ for almost all $t\in [0,T]$ and all $\varphi\in\Vspace.$  
\end{remark}

 Let $g: \mathbb{R}^n \times  \Omega\rightarrow(-\infty,+\infty]$  be a normal integrand (see ({\bf H})(iv) in Subsection \ref{notation}) and, for $p<\infty,$ consider the functional $\mathcal{J}$ (see \eqref{calJ}) over the space $L^p(\mu,\Omega;\mathbb{R}^n).$
The following result is a characterization of the lower semi-continuity of $\mathcal{J}.$
\begin{theorem}\cite{LUCCHETTIPATRONE}\label{chlsm}
Let ${\mathcal J}$ be not identically equal to $+\infty$ and let $p<\infty$. If $g$ is a normal integrand, then the following properties are equivalent:
\begin{itemize}
    \item[{\em(1)}] $\mathcal{J}$ is strongly lower semi-continuous on $L^p(\mu,\Omega;\mathbb{R}^n)$ and $$\mathcal{J}(\varphi)>-\infty \quad \forall \ \varphi \in L^p(\mu,\Omega;\mathbb{R}^n).$$
    \item[{\em(2)}] $\mathcal{J}: L^p(\mu,\Omega;\mathbb{R}^n) \rightarrow(-\infty,+\infty]$.
    \item[{\em (3)}] $\mathcal{J}:L^p(\mu,\Omega;\mathbb{R}^n)\rightarrow[-\infty,+\infty]$ and there exist $k \in \mathbb{R}, h \geq 0$ such that
    $$
    \mathcal{J}(\varphi) \geq k-h\|\varphi(\cdot)\|_{L^p(\mu,\Omega;\mathbb{R}^n) }^p \quad \text { for every } \,\varphi \in L^p(\mu,\Omega;\mathbb{R}^n) .
    $$
    \item[{\em (4)}] There exist $a \in L^1(\Omega), b \geq 0$ such that$$g(\varphi, \omega) \geq a(\omega)-b|\varphi|^p \quad \text { for every } \, (\varphi,\omega) \in \mathbb{R}^n\times \Omega. $$
\end{itemize}
If $g$ is not a normal integrand, then the following hold: (2) $\Leftrightarrow$ (3) $\Leftrightarrow$ (4) and (1) $\Rightarrow$ (2). 
\end{theorem}

\section{Compactness results}
\label{AppCompactness}

The following is a criterion for the relative compactness in the Banach space $$ L^p(\mu,\Omega ):=\Big\{f:\Omega\rightarrow\cR:\int_{\Omega}|f(\omega)|^p\,d\mu(\omega)<\infty\Big\}.$$
\begin{theorem}\label{compactL2}\cite{Agnieszka} Let $\Omega$ be a metric space equipped with a finite Borel measure $\mu$ satisfying condition \eqref{Hyp.mu}. Then, every bounded sequence $\left\{f_k\right\} \subset L^p(\mu,\Omega ),$ with $1 \leq p<\infty$ such that
$$\sup _k \int_{\Omega}\left|f_k(\omega)-\left(f_k\right)_{\mathcal{B}_{r}(\omega)}\right|^p d \mu(\omega) \stackrel{r \rightarrow 0}{\longrightarrow} 0$$ where $ (f)_{\mathcal{B}_{r}(\omega)}:=\frac{1}{\mu(\mathcal{B}_{r}(\omega))}\int_{\mathcal{B}_{r}(\omega)}f(\omega') \, d\mu(\omega'),$ is relatively compact in $L^p(\mu,\Omega)$.
\end{theorem}

The following result is a consequence of the Arzelà - Ascoli and Alaoglu Theorems.

\begin{theorem}\label{convergences}\cite[p. 13]{AubCelBook}
    Let us consider a sequence of absolutely continuous functions $x_k(\cdot)$ from an interval $I$ of $\mathbb{R}$ to a Banach space $X$ satisfying 
 $$\left\{\begin{array}{l}
 
 \text { i) } \text{for all } t \in I, \, \left\{x_k(t)\right\}_k \text { is a relatively compact subset of } X,
 \\ \\ 
 \text { ii) there exists a positive function } \, c(\cdot) \in L^1(I) \text { such that}
 \\
 \hspace{0.6cm}\text{for almost all } t \in I, \, \, \left\|x_k^{\prime}(t)\right\| \leq c(t).
\end{array}\right.$$
Then there exists a subsequence (again denoted by $x_k(\cdot)$) converging to an absolutely continuous function $x(\cdot)$ from $I$ to $X$ in the sense that
\\

 $\left\{\begin{array}{l}
 \text { i) } \,\, x_k(\cdot) \, \,\text{converges uniformly to} \, \,x(\cdot) \, \text{over compact subsets of} \, \,I,
 \\ \\
 \text { ii)} \,\, x_k^{\prime}(\cdot) \ \text{converges weakly to} \, \, x^{\prime}(\cdot) \, \, \text{in} \, \,  L^{1}(I,X).
\end{array}\right.$
\end{theorem}

\bibliographystyle{plain}
\bibliography{generalsoledad}

\end{document}

%% file: macro.tex
\if {

}{{\caption}}
 \fi

\def\eps{\varepsilon}

\def\1B{{\bf  1}}

\newcommand{\cR}{\mathbb{R}}

\newcommand\be{\begin{equation}}
\newcommand\ee{\end{equation}}
\newcommand{\benl}{\begin{equation*}}
\newcommand{\eenl}{\end{equation*}}
\newcommand\ba{\begin{array}}
\newcommand\ea{\end{array}}
\newcommand{\bean}{\begin{eqnarray*}}
\newcommand{\eean}{\end{eqnarray*}}

\def\ds{\displaystyle}
